\documentclass[10pt,reqno]{amsart}
\usepackage{bbm}
\usepackage{tabu}
\usepackage{amsmath}
\usepackage{mathrsfs}
\usepackage{bm}
\usepackage{amsfonts} 
\usepackage{graphicx,latexsym,bm,amsmath,amssymb,verbatim,multicol,lscape}
\usepackage{enumitem}
\newtheorem{thm}{Theorem} [section]
\newtheorem{prob}[thm]{Problem}

\newtheorem{lem}[thm]{Lemma}

\theoremstyle{definition}

\theoremstyle{remark}

\newtheorem{exa}[thm]{Example}

\numberwithin{equation}{section}

\newcommand{\fq}{{\mathbb F}_{q}}

\newcommand{\fp}{{\mathbb F}_{p}}

\newcommand{\rmv}[1]{}

\def\<{\left\langle}
\def\>{\right\rangle}

\begin{document}

\title[The inverse stability of Artin-Schreier polynomials ]
{The inverse stability of Artin-Schreier polynomials over finite fields}%
\author{Kaimin Cheng}
\address{School of Mathematics and Information, China West Normal University, Nanchong, 637002, P. R. China}
\email{ckm20@126.com}
\date{\today}
\begin{abstract}
Let $p$ be a prime number and $q$ a power of $p$. Let $\fq$ be the finite field with $q$ elements. For a positive integer $n$ and a polynomial $\varphi(X)\in\fq[X]$, let $d_{n,\varphi}(X)$ denote the denominator of the $n$th iterate of $\frac{1}{\varphi(X)}$. The polynomial $\varphi(X)$ is said to be inversely stable over $\fq$ if all polynomials $d_{n,\varphi}(X)$ are irreducible polynomial over $\fq$ and distinct. In this paper, we characterize a class of inversely stable polynomials over $\fq$. More precisely, for $\varphi(X)=X^{p^t}+aX+b\in\fq[X]$ with $t$ being a positive integer, we provide a sufficient and necessary condition for $\varphi(X)$ to be inversely stable over $\fq$.
\end{abstract}

\maketitle
\section{Introduction}
Constructing irreducible polynomials over finite fields is a central problem in finite field theory. Over the past few decades, numerous approaches have been developed to address this issue, particularly focusing on compositions and iterations of polynomials. One approach explores finding polynomials $f(X)$ and $g(X)$ such that $f(g(X))$ is an irreducible polynomial (see \cite{[BGL],[CBB],[CH],[Coh],[KK]}). Another direction of research is identifying polynomials $h$ such that all of its iterates remain irreducible over a given finite field. Such polynomials are called \textit{stable} over the finite field. Despite considerable attention to this topic (see \cite{[AF],[ALOS],[FPR],[GN],[JB]}), results remain sparse and the problem is far from being completely understood. For further discussions on different constructions of irreducible polynomials, one can refer to \cite{[BGM],[FMS],[Gao1],[Gao2],[GP],[GHP],[GK]} .

Let $q$ be a power of a prime number $p$, and let $\fq$ denote the finite field of order $q$. Consider a rational function $\Phi(x)=\frac{1}{\varphi(X)}$, where $\varphi(X)\in\fq[X]$. Denote by $\Phi^{(n)}(X)$ the $n$th iterate of $\Phi(X)$, i.e.,
$$\Phi^{(0)}(X)=X,\ \Phi^{(n)}(X)=\Phi(\Phi^{(n-1)}(X)),\ n=1,2,\ldots.$$
Let $d_{n,\varphi}(X)$ be the denominator of $\Phi^{(n)}(X)$. We say that $\varphi(X)$ is an \textit{inversely stable} polynomial over $\fq$ if all polynomials in the sequence $\{d_{n,\varphi}\}_{n=1}^{\infty}$ are irreducible over $\fq$ and are distinct. 

The \textit{Artin-Schreier polynomials} have the form $X^{p^t} + aX + b$, where $t$ is a positive integer and $a, b \in\fq$. Recently, the author \cite{[Che]} identified a class of inversely stable Artin-Schreier polynomials over the prime field $\fp$ and posed the following problem.
\begin{prob}\label{pro1.1}
Let $q=p^e$ with $e$ a positive integer. Characterize the conditions under which the polynomial $f(X)=X^{p^t}+aX+b\in\fq[X]$, where $t$ is a positive integer, is inversely stable over $\fq$.
\end{prob} 
In the case $e=t=1$, the author provided a complete solution to Problem \ref{pro1.1} (see \cite[Theorem 1.1]{[Che]}). However, for $e=t\ge 2$, it was shown that $f(X)=X^{p^t}+aX+b\in\fq[X]$ is never inversely stable over $\fq$. In this paper, we shall provide a solution to Problem \ref{pro1.1} for the case $e > t=1$. 

Let ${\rm Tr}_{\fq}$ be the absolute trace function from $\fq$ to $\fp$. For an element $\xi\in\fq$ with ${\rm Tr}_{\fq}(\xi)\ne 0$, define three sequences $\{a_n\},\{c_n\}$ and $\{d_n\}$ of $\fq$ as follows:
$$a_1=\xi,\ c_1=1,\ d_1=0;\  a_2=-1,\ c_2=\xi,\ d_2=-1;$$
and for $n\ge 2$,
\begin{align}\label{c1-1}
\begin{cases}
a_{n+1}=-a_nd_n,\\
c_{n+1}=c_n^2(\xi-(c_n^{-1}d_n)^p+c_n^{-1}d_n),\\
d_{n+1}=-c_n^2.
\end{cases}
\end{align}
Since none of the elements in $\{c_n\}$ is zero (as ${\rm Tr}_{\fq}(\xi)\ne 0$), the sequences are well defined. Now we present the main result of this paper.
\begin{thm}\label{thm1.2}
Let $p$ be a prime, and $q=p^e$ with $e$ a positive integer. Let $g(X)=X^p-X+\xi\in\fq[X]$ with ${\rm Tr}_{\fq}(\xi)\ne 0$. Then, $g(X)$ is inversely stable over $\fq$ if and only if ${\rm Tr}_{\fq}(a_nc_n^{-1})\ne 0$ for every positive integer $n$, where the sequences $\{a_n\}$ and $\{c_n\}$ are defined as in (\ref{c1-1}).
\end{thm}
Clearly, for a given polynomial $g(X)=X^p-X+\xi\in\fq[X]$ with ${\rm Tr}_{\fq}(\xi)\ne 0$, Theorem \ref{thm1.2} provides an efficient criterion to determine whether $g(X)$ is inversely stable over $\fq$. The conditions presented in the theorem allow straightforward verification of inverse stability by examining the absolute traces of specific sequences derived from the polynomial $g(X)$. Here are some examples.
\begin{exa}\label{exa1.3}
Let $p$ be a prime, and $q=p^e$ with $e$ a positive integer. Let $\xi$ be any nonzero element of $\fp$. It is evident that all the terms in the sequences $\{a_n\}$, $\{c_n\}$ and $\{d_n\}$ lie in $\fp$. We then know that $c_1=1$, and $c_{n+1}=\xi c_n^2$ for any $n\ge 2$, which leads to the expression $c_n=\xi^{2^{n-1}-1}$ for all $n\ge 2$. Consequently, $d_n=-c_{n-1}^2=-\xi^{2^{n-1}-2}$ for any $n\ge 2$. Furthermore, using the recurrence relation $a_{n+1}=-a_nd_n$, $n=2,3,\ldots$, we compute that 
$$a_n=\frac{a_n}{a_{n-1}}\cdot\frac{a_{n-1}}{a_{n-2}}\cdots\frac{a_3}{a_{2}}\cdot a_2=(-d_{n-1})\cdot(-d_{n-2})\cdots(-d_2)\cdot a_2=-\xi^{2^{n-1}-2n+2}$$
for any $n\ge 3$. Therefore, we derive that
$a_nc_n^{-1}=\xi$ if $n=1$, and $a_nc_n^{-1}=-\xi^{3-2n}$ if $n\ge 2$. Thus, we have
$${\rm Tr}_{\fq}(a_nc_n^{-1})=\begin{cases}
e\xi,&\text{if}\ n=1,\\
-e\xi^{3-2n},&\text{if}\ n\ge2.
\end{cases}
$$
Therefore, by Theorem \ref{thm1.2}, $g(X)$ is inversely stable over $\fq$ if and only if $p\nmid e$.
\end{exa}

\begin{exa}\label{exa1.4}
Let $p=3$ and $q=9$. Take $\xi=w$, where $w$ is a primitive element of $\mathbb{F}_9$ and $X^2+2X+2$ is the minimal polynomial over $\mathbb{F}_3$. By direct computations, we have the following:
\begin{align*}
&a_1=w,c_1=1,d_1=0,a_1c_1^{-1}=w,{\rm Tr}_{\mathbb{F}_9}(a_1c_1^{-1})=1,\\
&a_2=2,c_2=w,d_2=2,a_2c_2^{-1}=2w+1,{\rm Tr}_{\mathbb{F}_9}(a_2c_2^{-1})=1,\
\end{align*}
and for $n\ge 1$
$$\begin{cases}
a_{3n}=2,c_{3n}=2w,d_{3n}=2w+2,a_{3n}c_{3n}^{-1}=w+2,{\rm Tr}_{\mathbb{F}_9}(a_{3n}c_{3n}^{-1})=2,\\
a_{3n+1}=2w+2,c_{3n+1}=2w+2,d_{3n+1}=2w+2,a_{3n+1}c_{3n+1}^{-1}=1,{\rm Tr}_{\mathbb{F}_9}(a_{3n+1}c_{3n+1}^{-1})=2,\\
a_{3n+2}=1,c_{3n+2}=2w,d_{3n+2}=1,a_{3n+2}c_{3n+2}^{-1}=2w+1,{\rm Tr}_{\mathbb{F}_9}(a_{3n+2}c_{3n+2}^{-1})=1.
\end{cases}$$
Thus, from Theorem \ref{thm1.2} we conclude that $X^3-X+w$ is an inversely stable polynomial over $\mathbb{F}_9$.
\end{exa}

\begin{exa}\label{exa1.5}
Let $p=5$ and $q=25$. Take $\xi=w$, where $w$ is a primitive element of $\mathbb{F}_{25}$ and $X^2+4X+2$ is the minimal polynomial over $\mathbb{F}_5$. It is computed that $a_8=3w+3$, $c_8=w$, $a_8c_8^{-1}=w+2$ and ${\rm Tr}_{\mathbb{F}_{25}}(a_{8}c_{8}^{-1})=0$. Therefore, by Theorem \ref{thm1.2}, we conclude that $X^5-X+w$ is not inversely stable over $\mathbb{F}_{25}$.
\end{exa}

Theorem \ref{thm1.2} offers a solution to Problem \ref{pro1.1} for the case $e > t = 1$. This case is central, as solutions to all other cases of Problem \ref{pro1.1} can be obtained using the same method as that for the case $e > t = 1$ (see Section 4). Consequently, Problem \ref{pro1.1} is fully resolved in this sense.

The paper is organized as follows: Section 2 presents a series of lemmas that establish the foundation for the main results. Section 3 focuses on proving Theorem \ref{thm1.2}. Section 4 provides some remarks.

\section{Lemmas}
In this section, we introduce several useful lemmas that will form the foundation for proving our main results. Let $K$ be a finite field, and $L$ be an extension field of degree $m$. For $\alpha\in L$, denote by ${\rm Tr}_{L/K}(\alpha)$ the trace of $\alpha$ from $L$ to $K$, that is,
$${\rm Tr}_{L/K}(\alpha)=\alpha+\alpha^{|K|}+\alpha^{|K|^2}+\cdots+\alpha^{|K|^{m-1}},$$
where $|K|$ is the order of $K$. Additionally, the absolute trace of $\alpha$, denoted as ${\rm Tr}_{L}(\alpha)$, is defined by ${\rm Tr}_{L/\fp}(\alpha)$, where $\fp$ is the prime field of $K$.

Throughout this section, we assume that $K$ is a finite field with characteristic $p$. The first two lemmas follow immediately from basic properties of traces and irreducible polynomials.
\begin{lem}\label{lem2.1}\cite[Equation (2.2)]{[LN]}
Let $\alpha$ be an element of an extension field of $K$. Then the coefficient of the second highest term of the minimal polynomial of $\alpha$ over $K$ is equal to $-{\rm Tr}_{K(\alpha)/K}(\alpha)$.
\end{lem}

\begin{lem}\label{lem2.2}
Let $f(X)$ be a polynomial over $K$ of degree $m$. Define $g(X):=X^mf(\frac{1}{X})$. Then $f(X)$ is irreducible over $K$ if and only if $g(X)$ is irreducible over $K$.
\end{lem}

\begin{lem}\label{lem2.3}
Let $f(X)=X^p-X+\xi\in K[X]$ be irreducible over $K$, and  let $\gamma$ be a root in an extension field of $K$. For any $a,b,c,d\in K$ with $c$ and $d$ not both zero, the following hold:
\begin{enumerate}[label=(\alph*), left=0pt]
 \item If $c=0$, then ${\rm Tr}_{K(\gamma)/K}(\frac{a\gamma+b}{c\gamma+d})=0$ if $p\ge 3$, and ${\rm Tr}_{K(\gamma)/K}(\frac{a\gamma+b}{c\gamma+d})=\frac{a}{d}$ if $p=2$.
 \item If $c\ne 0$, then 
$${\rm Tr}_{K(\gamma)/K}\left(\frac{a\gamma+b}{c\gamma+d}\right)=\frac{bc-ad}{c^2(\xi-(\frac{d}{c})^p+\frac{d}{c})}.$$
\end{enumerate}
\end{lem}
\begin{proof}
The result of Item (a) follows immediately from the linearity of trace and Lemma \ref{lem2.1}. For Item (b), assume $c\ne0$. Using the linearity of trace, one deduces that
\begin{align}\label{c2-1}
{\rm Tr}_{K(\gamma)/K}\left(\frac{a\gamma+b}{c\gamma+d}\right)&=
{\rm Tr}_{K(\gamma)/K}\left(\frac{a}{c}+\frac{bc^{-1}-adc^{-2}}{\gamma+c^{-1}d}\right)\nonumber\\
&={\rm Tr}_{K(\gamma)/K}\left(\frac{a}{c}\right)+{\rm Tr}_{K(\gamma)/K}\left(\frac{bc^{-1}-adc^{-2}}{\gamma+c^{-1}d}\right)\nonumber\\
&=c^{-2}(bc-ad){\rm Tr}_{K(\gamma)/K}\left(\frac{1}{\gamma+c^{-1}d}\right).
\end{align}
Note that $f(X)$ is the minimal polynomial of $\gamma$ over $K$. It implies that 
$$f(X-c^{-1}d)=(X-c^{-1}d)^p-(X-c^{-1}d)+\xi=X^p-X+\xi-\left(\frac{d}{c}\right)^p+\frac{d}{c}$$
 is the minimal polynomial of $\gamma+c^{-1}d$ over $K$. By Lemma \ref{lem2.2} we then know that 
 $$X^pf\left(\frac{1}{X}-c^{-1}d\right)=\left(\xi-\left(\frac{d}{c}\right)^p+\frac{d}{c}\right)X^p-X^{p-1}+1$$
is irreducible over $K$, and then $X^p-(\xi-(\frac{d}{c})^p+\frac{d}{c})^{-1}X^{p-1}+(\xi-(\frac{d}{c})^p+\frac{d}{c})^{-1}$ is the minimal polynomial of $(\gamma+c^{-1}d)^{-1}$ over $K$. It follows from Lemma \ref{lem2.1} that 
$${\rm Tr}_{K(\gamma)/K}\left(\frac{1}{\gamma+c^{-1}d}\right)=\left(\xi-\left(\frac{d}{c}\right)^p+\frac{d}{c}\right)^{-1}.$$
Putting this into (\ref{c2-1}), we derive the desired result.
\end{proof}

\begin{lem}\label{lem2.4}\cite[Corollary 3.79]{[LN]}
For $f(X)=X^p-X+\xi\in K[X]$, we have that $f(X)$ is irreducible over $K$ if and only if ${\rm Tr}_{K}(\xi)\ne 0$.
\end{lem}

\begin{lem}\label{lem2.5}
Let $\overline{K}$ be the algebraic closure of $K$. Let $\infty$ be the point at infinity, i.e., $\infty=\frac{1}{0}$. For a polynomial $g(X)=X^p-X+\xi\in K[X]$, consider the rational map $G(X):=\frac{1}{g(X)}$ from $\overline{K}\cup\{\infty\}$ to itself. Then both of the following hold.
\begin{enumerate}[label=(\alph*), left=0pt]
    \item The map $G(X)$ is $p$-to-one, except that the preimage of $0$ is $\infty$.
    \item If ${\rm Tr}_{K}(\xi)\ne 0$, then $d_{n,g}(X)$ is a polynomial over $K$ of degree $p^n$ for every positive integer $n$, where $d_{n,g}(X)$ is the denominator of the $n$th iterate $G^{(n)}(X)$ of $G(X)$
\end{enumerate}
\end{lem}
\begin{proof}
First, let us prove Item (a). Clearly, for the map $G(X)$ the preimage of $0$ is $\infty$. Now let $\gamma$ be any fixed element of $\{\infty\}\cup\overline{K}\setminus\{0\}$. Note that the polynomial $g(X)-\frac{1}{\gamma}$ has no multiple roots, and then $g(X)-\frac{1}{\gamma}$ has exactly $p$ distinct roots in $\overline{K}$. Thus, the preimage set of $\gamma$ has order $p$, and the result of Item (a) follows.

Next, we prove Item (b). Let $n$ be a given positive integer. It is clear that $\deg(d_{n,g}(X))\le p^n$. Now we show that $\deg(d_{n,g}(X))\ge p^n$, completing the proof of Item (b). To do this, we define $n$ sets:
$$(G^{(i)})^{-1}(\infty):=\{x\in\overline{K}\cup\{\infty\}: G^{(i)}(x)=\infty\},\ 1\le i\le n.$$
We claim that $\infty\not\in(G^{(i)})^{-1}(\infty)$ for any $1\le i\le n$. Suppose that $\infty\in(G^{(i)})^{-1}(\infty)$ for some $i$, and let $k$ be the smallest positive integer with this property, that is,
$$G^{(1)}(\infty)\ne \infty,\ldots,G^{(k-1)}(\infty)\ne \infty,G^{(k)}(\infty)= \infty.$$
Let $\alpha_j:=G^{(j)}(\infty)$ for each $j$ with $1\le j\le k$. Then $\alpha_1,\ldots,\alpha_{k-1}$ are elements in $K$, and $g(\alpha_{k-1})=0$, implying that $g(X)$ has a root of $K$. However, by Lemma \ref{lem2.4}, $g(X)$ is irreducible over $K$ when ${\rm Tr}_{K}(\xi)\ne 0$, a contradiction. Therefore, $\infty\not\in(G^{(i)})^{-1}(\infty)$ for every $i$ with $1\le i\le n$. Now applying the result from Item (a), we see that
$$\#G^{(1)}(\infty)=p,\ \text{and}\ \#G^{(i+1)}(\infty)=p\#G^{(i)}(\infty)\ \text{for\ all}\ 1\le i\le n-1.$$
Thus, $\#G^{(n)}(\infty)=p^n$, which implies that $d_{n,g}(X)$ has exactly $p^n$ distinct roots in $\overline{K}$. Therefore, $\deg(d_{n,g}(X))\ge p^n$, completing the proof of Item (b). 

This concludes the proof of Lemma \ref{lem2.5}.
\end{proof}

\section{Proof of Theorem 1.2}
In this section, we primarily provide the proof of Theorem \ref{thm1.2}. \\
\textit{Proof of Theorem \ref{thm1.2}}. Let $g(X)=X^p-X+\xi\in\fq[x]$ with ${\rm Tr}_{\fq}(\xi)\ne 0$. Let $\{a_n\},\{c_n\}$ and $\{d_n\}$ be the sequences over $\fq$ defined as in (\ref{c1-1}). Let $G(X)=\frac{1}{g(X)}$, and for any positive integer $n$, let $G^{(n)}(X)$ be the $n$th iterate of $G(X)$. The denominator of $G^{(n)}(X)$ is denoted by $d_{n,g}(X)$. We aim to prove that $g(X)$ is inversely stable over $\fq$ if and only if ${\rm Tr}_{\fq}(a_nc_n^{-1})\ne 0$ for every positive integer $n$. 

First, we prove the necessity.  Let $g(X)=X^p-X+\xi$ be an inversely stable polynomial over $\fq$. We are going to prove that ${\rm Tr}_{\fq}(a_nc_n^{-1})\ne 0$ for every positive integer $n$. Let $n$ be any fixed positive integer. Note that $d_{n,g}(X)$ is irreducible over $\fq$, and it is of degree $p^n$ from Lemma \ref{lem2.5}. Let $\beta_n$ be a root of $d_{n,g}(X)$ in the algebraic closure of $\fq$. Then $\fq(\beta_n)$ is an extension field of degree $p^n$. For every integer $i$ with $1\le i\le n-1$, define $\beta_i:=G^{(n-i)}(\beta_n)$. One then readily finds that
\begin{align}\label{c3-1}
&\beta_1^p-\beta_1+\xi=0,\ \text{and}\\ \label{c3-2}
&\beta_i^p-\beta_i+\frac{\xi\beta_{i-1}-1}{\beta_{i-1}}=0\ \text{for\ any}\ 2\le i\le n.
\end{align}
Now define a sequence $\{K_i\}_{i=0}^n$ of finite fields by
$$K_0=\fq,\ \text{and}\ K_{i}=K_{i-1}(\beta_{i})\ \text{for}\ 1\le i\le n.$$
Let $h_1(X)=X^p-X+\xi$, and for $2\le j\le n$, $h_{j}(X)=X^p-X+\frac{\xi\beta_{j-1}-1}{\beta_{j-1}}$. Clearly, $h_i(X)\in K_{i-1}[X]$ and $h_i(\beta_i)=0$ for any $1\le i\le n$. It implies that $[K_i:K_{i-1}]\le p$ for each $1\le i\le n$. So, we obtain that 
$$p^n=[\fq(\beta_n):\fq]=[K_n:K_0]=\prod_{i=1}^n[K_i:K_{i-1}]\le p^n,$$
and then $[K_i:K_{i-1}]=p$ for any $1\le i\le n$. It follow that $h_i(X)$ is irreducible over $K_{i-1}$ for each $1\le i\le n$. In particular, $h_n(X)=X^p-X+\frac{\xi\beta_{n-1}-1}{\beta_{n-1}}$ is irreducible over $K_{n-1}$. Applying Lemma \ref{lem2.4}, we then have
$${\rm Tr}_{K_{n-1}}\left(\frac{\xi\beta_{n-1}-1}{\beta_{n-1}}\right)\ne 0.$$
On the other hand, noting that each $h_i(X)$ is irreducible over $K_{i-1}$, and employing Lemma \ref{lem2.3} again and again, we compute that
\begin{align*}
&{\rm Tr}_{{K_{n-1}}/K_{n-2}}\left(\frac{\xi\beta_{n-1}-1}{\beta_{n-1}}\right)=\frac{a_2\beta_{n-2}}{c_2\beta_{n-2}+d_2},\\
&{\rm Tr}_{{K_{n-2}}/K_{n-3}}\left(\frac{a_2\beta_{n-2}}{c_2\beta_{n-2}+d_2}\right)=\frac{a_3\beta_{n-3}}{c_3\beta_{n-3}+d_3},\\
&\vdots\\
&{\rm Tr}_{{K_{2}}/K_{1}}\left(\frac{a_{n-2}\beta_{2}}{c_{n-2}\beta_{2}+d_{n-2}}\right)=\frac{a_{n-1}\beta_{1}}{c_{n-1}\beta_{1}+d_{n-1}},\\
&{\rm Tr}_{{K_{1}}/K_{0}}\left(\frac{a_{n-1}\beta_{1}}{c_{n-1}\beta_{1}+d_{n-1}}\right)=a_nc_n^{-1}.
\end{align*}
Therefore, by the transitivity of trace we arrive at 
$${\rm Tr}_{K_0}(a_nc_n^{-1})={\rm Tr}_{K_{n-1}}\left(\frac{\xi\beta_{n-1}-1}{\beta_{n-1}}\right)\ne 0.$$
The necessity is proved.

Next, we focus on the sufficiency. Let $\xi\in\fq$ satisfy that ${\rm Tr}_{\fq}(a_nc_n^{-1})\ne 0$ for any $n\ge 1$. We shall prove that $g(X)=X^p-X+\xi$ is inversely stable over $\fq$. Let $n$ be any given positive integer. By Lemma \ref{lem2.5}, we know that $d_{n,g}(X)$ is of degree $p^n$, so we only need to prove that $d_{n,g}(X)$ is irreducible over $\fq$, which will be done in what follows. Let $\beta_i$, $K_i$ and $h_i(X)$ be defined as in the proof of necessity. First, we show that $h_{i}(X)$ is irreducible over $K_{i-1}$ for any $1\le i\le n$. Obviously, $h_1(X)$ is irreducible over $K_0$, and we need to prove that $h_{i}(X)$ is irreducible over $K_{i-1}$ for any $2\le i\le n$, which is equivalent to that
\begin{align}\label{c3-3}
{\rm Tr}_{K_{i-1}}\left(\frac{\xi\beta_{i-1}-1}{\beta_{i-1}}\right)\ne 0
\end{align}
for any $2\le i\le n$. Now we prove (\ref{c3-3}) by mathematical induction on $i$. 
\begin{itemize}[left=1.5em] 
\item For $i=2$, noting that $h_1(X)$ is irreducible over $K_0$, $h_1(\beta_1)=0$ and $K_1=K_0(\beta)$, by using Lemma \ref{lem2.3} one then checks that
$${\rm Tr}_{K_{1}/K_0}\left(\frac{\xi\beta_{1}-1}{\beta_{1}}\right)=a_2c_2^{-1},$$
and then
$${\rm Tr}_{K_{1}}\left(\frac{\xi\beta_{1}-1}{\beta_{1}}\right)={\rm Tr}_{K_0}(a_2c_2^{-1})\ne 0.$$
It is to say that (\ref{c3-3}) is true for $i=2$.
\item Assume that (\ref{c3-3}) is true for all $i$ with $i\le s-1$. It implies from Lemma \ref{lem2.4} that $h_i(X)$ is irreducible over $K_{i-1}$ for each $1\le i\le s-1$.  Since $K_i=K_{i-1}(\beta_i)$ and $h_i(\beta_i)=0$ for any $1\le i\le s-1$, by using Lemma \ref{lem2.3} repeatedly we derive that
\begin{align*}
&{\rm Tr}_{{K_{s-1}}/K_{s-2}}\left(\frac{\xi\beta_{s-1}-1}{\beta_{s-1}}\right)=\frac{a_2\beta_{s-2}}{c_2\beta_{s-2}+d_2},\\
&{\rm Tr}_{{K_{s-2}}/K_{s-3}}\left(\frac{a_2\beta_{s-2}}{c_2\beta_{s-2}+d_2}\right)=\frac{a_3\beta_{s-3}}{c_3\beta_{s-3}+d_3},\\
&\vdots\\
&{\rm Tr}_{{K_{2}}/K_{1}}\left(\frac{a_{s-2}\beta_{2}}{c_{s-2}\beta_{2}+d_{s-2}}\right)=\frac{a_{s-1}\beta_{1}}{c_{s-1}\beta_{1}+d_{s-1}},\\
&{\rm Tr}_{{K_{1}}/K_{0}}\left(\frac{a_{s-1}\beta_{1}}{c_{s-1}\beta_{1}+d_{s-1}}\right)=a_sc_s^{-1}.
\end{align*}
It then follows from the transitivity of trace that
$${\rm Tr}_{K_{s-1}}\left(\frac{\xi\beta_{s-1}-1}{\beta_{s-1}}\right)={\rm Tr}_{K_{0}}(a_sc_s^{-1})\ne 0.$$
Hence, (\ref{c3-3}) is still true for $i=s$, and then (\ref{c3-3}) holds for any $2\le i\le n$.
\end{itemize}
Therefore, each $h_i(X)$ is the minimal polynomial of $\beta_i$ over $K_{i-1}$. It gives that $[K_i:K_{i-1}]=p$ for all $1\le i\le n$, and then we have 
$$[K_n:K_0]=\prod_{i=1}^n[K_i:K_{i-1}]=p^n.$$
From (\ref{c3-2}), one sees that $K_n=K_0(\beta_n)$, so that $\beta_n$ is of degree $p^n$ over $K_0$. Note that $d_{n,g}(X)$ is a polynomial over $K_0$ of degree $p^n$ and $d_{n,g}(\beta_n)=0$. Thus, we derive that $d_{n,g}(X)$ is irreducible over $K_0$, and the sufficiency is proved.

This completes the proof of Theorem \ref{thm1.2}.\hfill \qed

\section{Conclusion and Remarks}

In this paper, we provided a complete characterization of the inverse stability of Artin--Schreier polynomials of the form $f(X) = X^{p^t} + aX + b \in \mathbb{F}_q[X]$ for the essential case $t=1$, $a=-1$, as posed in Problem \ref{pro1.1}. Our main result, Theorem \ref{thm1.2}, establishes that the inverse stability of $g(X) = X^p - X + b \in \mathbb{F}_q[X]$ is equivalent to the sequence-based condition $\operatorname{Tr}_{\fq}(a_n c_n^{-1}) \neq 0$ for all $n \geq 1$. This criterion enables us to directly verify the inverse stability of $g(X)$ by examining the absolute traces of specific sequences derived from the polynomial $g(X)$.

To characterize the inverse stability of $f(X) = X^{p^t} + aX + b$, we require $t \geq 1$ and $a \neq 0$. When $a = 0$, $f(X)$ is reducible, as any $p$-polynomial over $\fq$ with no nonzero roots in $\fq$ permutes $\fq$. Here, a $p$-polynomial has the form $\sum_i c_i X^{p^i} \in \fq[X]$. In this paper, we only address Problem \ref{pro1.1} in the case of $t=1$ and $a=-1$, but this case is fundamental. In fact, a complete solution to Problem \ref{pro1.1} can be derived using the methods developed in this paper and the following observations, without repeating the proof process:
\begin{enumerate}
\item For $t=1$ and $a \neq 0$, a well-known result (see, for example, \cite[Theorem 10.12]{[Wan]}) states that $X^p + aX + b$ is irreducible over $\fq$ if and only if there exists $a_0 \in \fq^*$ such that $a = -a_0^{p-1}$ and $\operatorname{Tr}(b/a_0^p) \neq 0$. By applying the same methods as in this paper, a characterization similar to Theorem \ref{thm1.2} can be obtained.
\item For $t \geq 2$ and $a \neq 0$, a result by Agou \cite{[Ago]} indicates that a necessary condition for $f(X)$ to be irreducible over $\fq$ is $t = p = 2$. Thus, it suffices to characterize the inverse stability of $X^4 + aX + b$ over $\mathbb{F}_{2^e}$. Moreover, according to Agou’s result, $X^4 + aX + b$ is irreducible over $\mathbb{F}
_{2^e}$ if and only if $e$ is odd, $a = a_0^3$, and $\operatorname{Tr}(b/a_0^2) \neq 0$ for some $a_0 \in \mathbb{F}_{2^e}^*$. Hence, the methods in this paper can be used to characterize the inverse stability of $X^4 + aX + b$ over $\mathbb{F}_{2^e}$.
\end{enumerate}

Finally, we note that when the sequences $\{a_n\}_{n \geq 1}$ and $\{c_n\}_{n\geq1}$ are periodic, as shown in Example \ref{exa1.4}, the condition $\operatorname{Tr}_{\fq}(a_n c_n^{-1}) \neq 0$ can be checked by computing the trace for a finite number of terms within one period, rendering the criterion practical. In contrast, for non-periodic sequences, a key difficulty emerges: verifying $\operatorname{Tr}_{\fq}(a_n c_n^{-1}) \neq 0$ for increasingly large $n$ does not guarantee inverse stability, as the condition must hold for all $n$. This requires examining ever-larger $n$ without a finite stopping point, presenting a significant practical challenge.

\section*{acknowledgment}
The author expresses gratitude to the anonymous referees for their insightful and valuable comments, which significantly enhanced the paper's presentation.

This work was conducted during the author’s academic visit to RICAM, Austrian Academy of Sciences. The author sincerely thanks Professor Arne Winterhof at RICAM for his thorough review of the manuscript and for offering invaluable suggestions.

This research was partially supported by the China Scholarship Council Fund (Grant No.\ 202301010002) and the Scientific Research Innovation Team Project of China West Normal University (Grant No.\ KCXTD2024-7).

%

\end{document}